\documentclass[a4paper,12pt]{article}
\usepackage{amssymb}
\usepackage{amsmath,amssymb,amsthm,graphics,latexsym,amsfonts}
\usepackage{fancyhdr}
\usepackage{graphics}

\title{\Large Proof of a conjecture on the zero forcing number of a graph \thanks {Research supported by
NSFC (No. 11161046) and by Xingjiang Talent Youth Project (No. 2013721012).}}
\author{ Leihao Lu, {Baoyindureng Wu \footnote{Corresponding author.
Email: wubaoyin@hotmail.com (B. Wu) }}, Zixing Tang\\
\small  College of Mathematics and System Sciences, Xinjiang
University \\ \small  Urumqi, Xinjiang 830046, P.R. China \\}
\date{}

\usepackage{indentfirst}
\newtheorem{theorem}{Theorem}[section]
\newtheorem{lemma}[theorem]{Lemma}

\begin{document}
\maketitle

{\small \noindent{\bfseries Abstract} Amos et al. (Discrete Appl.
Math. 181 (2015) 1-10) introduced the notion of the $k$-forcing
number of graph for a positive integer $k$ as the generalization of
the zero forcing number of a graph. The $k$-forcing number of a
simple graph $G$, denoted by $F_k(G)$, is the minimum number of
vertices that need to be initially colored so that all vertices
eventually become colored during the discrete dynamical process by
the following rule. Starting from an initial set of colored vertices
and stopping when all vertices are colored: if a colored vertex has
at most $k$ non-colored neighbors, then each of its non-colored
neighbors become colored. Particulary, $F_1(G)$ is a widely studied
invariant with close connection to the maximum nullity of a graph,
under the name of the zero forcing number, denoted by $Z(G)$. Among
other things, the authors proved that for a connected graph $G$ of
order $n$ with $\Delta=\Delta(G)\geq 2$, $Z(G)\leq
\frac{(\Delta-2)n+2}{\Delta-1}$, and this inequality is sharp.
Moreover, they conjectured that
$Z(G)=\frac{(\Delta-2)n+2}{\Delta-1}$ if and only if $G=C_n$,
$G=K_{\Delta+1}$ or $G=K_{\Delta, \Delta}$. In this note, we show
the above conjecture is true.

\vspace{1mm}\noindent{\bfseries Keywords}: Zero forcing set; Zero
forcing number; Rank; Nullity

\section{\large Introduction}
We consider undirected finite simple connected graphs only. For
notation and terminology not defined here, we refer to \cite{Bo}.
For a graph $G = (V(G), E(G))$, $|V(G)|$ and $|E(G)|$ are its order
and size, respectively. For a vertex $v\in V(G)$, the neighborhood
$N(v)$ of $v$ is defined as the set of vertices adjacent to $v$. The
degree $d_G(v)$ of $v$ is the number of edges incident with $v$ in
$G$. The minimum and maximum degrees of a vertex in a graph $G$ are
denoted $\delta(G)$ and $\Delta(G)$, respectively. Let $S\subseteq
V(G)$. Denote the set of the edges between $S$ and $\overline{S}$ by
$E(S,\overline{S})$, and let
$e(S,\overline{S})=|E(S,\overline{S})|$. The subgraph induced by
$S$, denoted by $G[S]$, is the graph with vertex set $S$, in which
two vertices $x$ and $y$ are adjacent if and only if they are
adjacent in $G$. As usual, for a positive integer $n\geq 1$, $K_n$
and $K_{n,n}$ denote respectively the complete graph of order $n$
and the complete bipartite graph with $n$ vertices in its each part;
$C_m$ denote the cycle of order $m$ for an integer $m\geq 3$.

Next, we follow the definition by Amos et al. \cite{Amos}. Let $k$
be a positive integer and $G$ a graph. A set $S\subseteq V(G)$ is a
{\it $k$-forcing set} if, when its vertices are initially colored -
while the remaining vertices are initially non-colored - and the
graph is subjected to the following color change rule, all of
vertices in $G$ will eventually become colored. A colored vertex
with at most $k$ non-colored neighbors will cause each the
non-colored neighbor to become colored. The {\it $k$-forcing number}
of $G$, denoted by $F_k(G)$, is the cardinality of the smallest
$k$-forcing set. If a vertex $u$ cause a vertex $v$ change colors
during the $k$-forcing process, we say that {\it $u$ $k$-forces $v$}
(in particular, $u$ forces $v$ when $k=1$).

This concept generalizes a widely studied notion of the zero forcing
number $Z(G)$ of a graph $G$. Indeed, $F_1(G)=Z(G)$. Barioli et al.
\cite{Bar} and Burgarth et al. \cite{Bu} introduced independently
the concepts of zero forcing set and zero forcing number of a graph.
In \cite{Bar}, it is introduced to bound the maximum nullity $M(G)$
of a graph. Namely, for a graph $G$ whose vertices are labeled from
1 to $n$, $M(G)$ denote the maximum nullity over all symmetric real
valued matrices where, for $i\neq j$, the $ij$th entry in nonzero if
and only if $ij$ is an edge in $G$. Then, $M(G)\leq Z(G)$ for any
graph $G$. For the more results on the relation between the relation
of the maximum nullity and the zero forcing number of a graph, we
refer to \cite{Bari, Ber, Edh, Eroh1, Eroh2, Eroh3, Fallat, Hog,
Mey}. In \cite{Bu}, the zero forcing set of a graph has been used in
order to study the controllability of quantum systems. Aazami
\cite{Aa} proved the NP-hardness of computing the zero forcing
number of a graph, using a reduction from the Directed Hamiltonian
Cycle problem.

\vspace{3mm} Amos et al. \cite{Amos} generalized the concept of zero
forcing number of a graph to the $k$-forcing number of a graph for
an integer $k\geq 1$ and proved that for a connected graph $G$ of
order $n$ with $\Delta=\Delta(G)\geq 2$, $Z(G)\leq
\frac{(\Delta-2)n+2}{\Delta-1}$, and this inequality is sharp.
Moreover, they posed the following conjecture.

\vspace{3mm}\noindent{\bf Conjecture (Amos et al. \cite{Amos}).} Let
$G$ be a connected graph with $\Delta\geq 2$. Then
$$Z(G)=\frac{(\Delta-2)n+2} {\Delta -1},$$ if and only if $G=C_n$,
$G=K_{\Delta+1}$ or $G=K_{\Delta,\Delta}$.

\vspace{3mm} In this note, we confirm the validity of the above
conjecture.

\section{\large Some results on $Z(G)$}



A {\it $k$-dominating set} of a graph $G$ is a set $D$ of vertices
such that every vertex not in $D$ is adjacent to at least $k$
vertices in $D$.

\begin{lemma} ( Lemma 4.1 in \cite{Amos}) Let $k$ be a positive
integer and $G=(V, E)$ be a $k$-connected graph with $n>k$. If $S$
is a smallest $k$-forcing set such that the subgraph induced by
$V\setminus S$ is connected, then $V\setminus S$ is a connected
$k$-dominating set of $G$.
\end{lemma}

\begin{theorem}(\cite{Amos})
Let $k$ be positive integer and let $G=(V,E)$ be a $k$-connected
graph with $n>k$ vertices and $\Delta\geq 2$. Then
$$F_k(G)\leq \frac{(\Delta-2)n+2}{\Delta+k-2},$$ and this inequality is sharp.
\end{theorem}

\begin{theorem}(Corollary 3.1 in \cite{Caro})
Let $G$ be a connected graph of order $n$ with maximum degree
$\Delta$ and minimum degree $\delta$. Then $$Z(G)\leq
\frac{(\Delta-2)n-(\Delta-\delta)+2}{\Delta-1}.$$
\end{theorem}

\begin{lemma} Let $T$ be a tree with exactly $k$ leaves. If $S$ is a
set of $k-1$ leaves of $T$, then $S$ is a zero forcing set of $T$.
\end{lemma}
\begin{proof}
The proof is by induction on $k$. If $k=2$, $T$ is path, and the
result clearly holds. Now assume that $k\geq 3$. Take a vertex $u\in
S$. Let $P$ be a maximal path of $T$ containing $u$ such that every
vertex $v$ on $P$ has degree at most two in $T$. Let $T'=T-V(P)$.
Note that $T'$ has exactly $k-1$ leaves. By the induction
hypothesis, $S'=S\setminus \{u\}$ is a zero forcing set of $T'$. So,
$S$ is a zero forcing set of $T$.
\end{proof}

\section{\large Main result}
\begin{theorem} Let $G$ be a connected graph with $\Delta\geq 2$. Then
$$Z(G)=\frac{(\Delta-2)n+2} {\Delta -1},$$ if and only if $G=C_n$,
$G=K_{\Delta+1}$ or $G=K_{\Delta,\Delta}$.
\end{theorem}
\begin{proof}
It is clear that $Z(C_n)=2$ for any $n\geq 3$,
$Z(K_{\Delta+1})=\Delta$, $Z(K_{\Delta, \Delta})=2\Delta-2$. Hence,
the sufficiency of theorem holds trivially.

To show the necessity, we assume that $G$ is a connected graph of
order $n$ with $\Delta\geq 2$ and $Z(G)=\frac{(\Delta-2)n+2} {\Delta
-1}$. By Theorem 2.3, $G$ is a $\Delta$-regular graph. If
$\Delta=2$, then $G=C_n$. In what follows, we assume that
$\Delta\geq 3$.

Let $S$ be a smallest zero forcing set of $G$ such that
$G[\overline{S}]$ is connected, where $\overline{S}=V\setminus S$.
Thus,
\begin{equation}
|S|\geq Z(G)=\frac{(\Delta-2)n+2} {\Delta -1}.
\end{equation}



\vspace{3mm}\noindent {\bf Claim 1.} Each vertex of $S$ has exactly
one neighbor in $\overline{S}$ and $G[\overline{S}]$ is a tree.

\begin{proof} By Lemma 2.1,

\begin{equation}
e(S, \overline{S})\geq |S|.
\end{equation}

On the other hand,

\begin{equation}
\begin{split}
e(S, \overline{S})
=&\sum_{v\in \overline{S}}(d(v)-d_{\overline{S}}(v))\\
=& \sum_{v\in \overline{S}}d(v)-\sum_{v\in \overline{S}}d_{\overline{S}}(v)  \\
\leq & \Delta|\overline{S}|-2(|\overline{S}|-1)\\
=& (\Delta-2)|\overline{S}|+2\\
=& (\Delta-2)(n-|S|)+2. \\
\end{split}
\end{equation}

Combining (2) and (3), we have
\begin{equation}
|S|\leq \frac{(\Delta-2)n+2} {\Delta -1}.
\end{equation}

Combining (1) and (4), we have
\begin{equation}
|S|=\frac{(\Delta-2)n+2} {\Delta -1}.
\end{equation}

From (5), we can conclude that $S$ is a smallest forcing set of $G$
and that each vertex of $S$ has exactly one neighbor in
$\overline{S}$ and $G[\overline{S}]$ is a tree.
\end{proof}

Note that
\begin{equation}
|\overline{S}|=n-|S|=\frac{n-2}{\Delta-1}.
\end{equation}

If $|\overline{S}|=1$, by (6), $\Delta=n-1$. Since $G$ is
$(n-1)$-regular, $G\cong K_n=K_{\Delta+1}$. Next we assume that
$|\overline{S}|\geq 2$ and let $x$ be a leaf of $G[\overline{S}]$
and $X=N(x)\cap S=\{x_1, \ldots, x_{\Delta-1}\}$.


%

\vspace{3mm}\noindent {\bf Claim 2.} $X$ is either an independent
set or a clique.

\begin{proof} We assume that $X$ is not an independent set, and
show that $X$ is a clique. Let $x_{1}, x_{2}\in X$ with $x_{1}x_{2}
\in E(G)$. Since $\Delta\geq 3$, there exists a neighbor $y_{1}$ of
$x_{1}$ in $S\setminus X$.

First we show that $y_1$ is adjacent to all vertices of $X$ in $G$.
To see this, suppose that there exists a vertex $x_{j}\in X$, where
$2\leq j\leq \Delta-1$, which is not adjacent to $y_{1}$. Since
$\Delta\geq 3$, by Claim 1, there exists a neighbor $y_j\in
S\setminus X$ of $x_j$ in $G$. Set $S'=S\cup\{x\}\setminus \{y_1 ,
y_j\}$. We can show that $S'$ is a zero forcing set of $G$. Observe
that all neighbors of $x_j$ but $y_j$ are initially colored. So, by
the color exchange rule, $y_j$ should be colored. Now, all neighbors
of $x_1$ but $y_1$ are colored. By the color exchange rule, $y_1$ is
forced to be colored. All vertices of $S$ are colored, and thus $S'$
is a zero forcing set of $G$. Since $|S'|<|S|$, which contradicts
the fact that $S$ is a minimum zero forcing set of $G$.

Next we show that $x_1$ is adjacent to all vertices of $X$ in $G$.
Suppose that this is not, and that $x_1x_j\notin E(G)$ for some
vertex $x_j\in X$. Set $S'=S\cup \{x\}\setminus \{x_1, y_1\}$. We
consider $x_j$. Note that all neighbors of $x_j$ but $y_1$ are
initially colored. By the color exchange rule, $y_1$ is colored.
Now, all neighbors of $x_2$ but $x_1$ are colored. By the color
exchange rule, $x_1$ is colored. Since $|S'|<|S|$, which contradicts
the fact that $S$ is a minimum zero forcing set of $G$.

Finally, by an argument similar to the above, one can prove that
$x_i$ is adjacent to every other vertex in $X$ for each $i\geq 2$.
Thus, $X$ is a clique of $G$.

This proves the claim.
\end{proof}

\vspace{3mm}\noindent {\bf Claim 3.} If $X$ is a clique of $G$, then
there exists a unique vertex $y$ in $S$ such that $N(y)=X\cup
\{z\}$, and $d_{G[\overline{S}]}(z)=\Delta-1$, where $z$ is the
unique neighbor of $y$ in $\overline{S}$.
\begin{proof} The first half of the assertion can be deduced from
the proof of Claim 2 (see the paragraph starting with ``First we
show''). We show $d_{G[\overline{S}]}(z)=\Delta-1$ by contradiction.
Suppose that $d_{G[\overline{S}]}(z)\neq \Delta-1$, and let $z'\in
N(z)\cap S$ and $z''\in N(z')\cap S$. Note that $y\neq z'$ and
$y\neq z''$. Set $S'=S\cup \{z\}\setminus \{z'', x_1\}$. By the
color exchange role, $z'$ forces $z''$, and $y$ forces $x_1$. Now
all vertices of $S$ are already colored. But, $|S'|<|S|$, a
contradiction.

\end{proof}

\vspace{3mm}\noindent {\bf Claim 4.} If $X$ is an independent set of
$G$, then $N(x_i)\cap S=N(x_j)\cap S$ for any two vertices $x_i, x_j
\in X$, and $N(x_i)\cap S$ is an independent set of $G$ with
cardinality $\Delta-1$. Moreover, if $z_i$ is a leaf of
$G[\overline{S}]$, where $z_i$ is the unique neighbor of $y_i\in
N(x_i)\cap S$, then $N(N(x_i)\cap S))\cap \overline{S}=\{z_i\}$.
\begin{proof}
By an argument similar to the proof of Claim 2 (see the paragraph
starting with ``First we show''), one can show that $N(x_i)\cap
S=N(x_j)\cap S$ for any two vertices $x_i, x_j \in N(x)\cap S$. By
contradiction, suppose that $y_j\in N(x_i)\cap S$ is not adjacent to
$z_i$ in $G$. Since $z_i$ is a leaf of $G[\overline{S}]$ and $G$ is
$\Delta$-regular, $z_i$ has a neighbor $z'\in S\setminus (X\cup
N(X))$ and $z'' \in N(z')\cap S$. Note that $z''\in S\setminus
(X\cup N(X))$. Set $S'=S\cup \{z_i\}\setminus \{z'', x_1\}$. By the
color exchange role, $z'$ forces $z''$, and then $y_j$ forces $x_1$.
Now all vertices of $S$ are already colored. But, $|S'|<|S|$, a
contradiction.
\end{proof}

Before proceeding, we recall the definition of bridge, which can be
find on the page 263 in \cite{Bo}. Let $H$ be a proper subgraph of a
connected graph $G$. The set $E(G)\setminus E(F)$ may be partitioned
into classes as follows.

$(i)$. For each component $F$ of $G-V(H)$, there is a class
consisting of the edges of $F$ together with the edges linking $F$
to $H$.

$(ii)$. Each remaining edge $e$ (that is, one which has both ends in
$V(H)$) defines a singleton class $\{e\}$.

The subgraphs of $G$ induced by these classes are the {\it bridges}
of $H$ in $G$. For a bridge $B$ of $H$, the elements of $V(B)\cap
V(H)$ are called its vertices of attachment to $H$; the remaining
vertices of $B$ are its internal vertices. A bridge is trivial if it
has not internal vertices.  A bridge with $k$ vertices of attachment
is called a $k$-bridge. Observe that bridges of $H$ can intersect
only in vertices of $H$.

\vspace{3mm}\noindent {\bf Claim 5.} Let $B_i$ be a bridge of
$G[\overline{S}]$ containing a leaf $z_i$ of $G[\overline{S}]$ for
$1\leq i\leq 2$. Then $B_1=B_2$ or $V(B_1) \cap V(B_2)=\emptyset$.

\begin{proof} By contradiction, suppose that $B_1\neq B_2$ and $V(B_1)\cap V(B_2)\neq
\emptyset$, and let $w\in V(B_1)\cap V(B_2)$. Let $w_1\in V(B_1)\cap
S$ and $w_2\in V(B_2)\cap S$. Take a vertex $w_1'\in N(w_1)\cap S$
and a vertex $w_2'\in N(w_2)\cap S$. Set $S'=S\cup
\{w\}\setminus\{w_1', w_2'\}$. In this case, $w_1$ forces $w_1'$ and
$w_2$ forces $w_2'$. Thus, all vertices of $S$ are colored. This
shows that $S'$ is zero forcing set of $G$, a contradiction.

\end{proof}

We consider the case when $|\overline{S}|=2$. Let
$\overline{S}=\{z_1, z_2\}$. Let $B_i$ be the bridge of $
G[\overline{S}]$ containing $z_i$ for $1\leq i\leq 2$. Since
$|V(B_i)\cap \overline{S}|\geq 2$ and $|\overline{S}|=2$, by Claim
5, $B_1=B_2$, which implies that $G\cong K_{\Delta, \Delta}$.

\vspace{3mm} Next, we complete the proof by showing that
$|\overline{S}|\geq 3$ is not possible. We consider the following
cases.


\vspace{3mm} \noindent {\bf Case 1.} $X$ is a clique of $G$.

Let $S'=S\setminus \{x_1\}$. We will show that $S'$ is a zero
forcing set of $G$. By Claim 5, each leaf $z$ of $G[\overline{S}]$
distinct from $x$ is forced to be colored in 1 by some vertex in
$S'$. Note that $T=G[\overline{S}\cup \{x_1\}]$ is a tree with
exactly $k$ leaves. By Lemma 2.4, $L\setminus \{x_1\}$ is a zero
forcing set of $T$, where $L$ is the set of leaves
$G[\overline{S}]$. This shows that $S'$ is a zero forcing set of
$G$, contradicting the choice of $S$.

\vspace{3mm}\noindent {\bf Case 2.} $N(x)\cap S$ is an independent
set of $G$ for each leaf $x$ of $G[\overline{S}]$.

Take a leaf $x$ of $G[\overline{S}]$, and let $N(x)\cap S=\{x_1,
\ldots, x_{\Delta-1}\}$. By Claim 3, we know that $N(x_i)\cap
S=N(x_j)\cap S$ for any two neighbors $x_i, x_j \in S$ in $G$, and
$N(x_i)\cap S$ is an independent set of $G$ with cardinality
$\Delta-1$. Let $N(x_i)\cap S=\{y_1, \ldots, y_{\Delta-1}\}$. Let
$z_i \in \overline{S}$ be the unique neighbor of $y_i$ in $G$.  By
Claim 4, we consider two subcases.

\vspace{3mm}\noindent {\bf Case 2.1.} $z_i$ is not a leaf of
$G[\overline{S}]$ for each $i\in \{1, \ldots, \Delta-1\}$.

\vspace{3mm} Let $S'=S\setminus \{x_1\}$. By an argument same as the
proof of tackling Case 1, one may obtain a contradiction by showing
that $S'$ is a zero forcing set of $G$.

\vspace{3mm}\noindent {\bf Case 2.2.} $z_i$ is a leaf of
$G[\overline{S}]$ and $z_j=z_i$ for each $j$ other than $i$.

\vspace{3mm} For the simplicity, let $z=z_i$. Let
$S'=S\setminus\{x_1\}$. Since $N(x)\cap S$ is an independent set of
$G$, $x$ is forced to colored in 1 by $x_2$. Note that $S'$ forces
to color all leaves of $G[\overline{S}]$ but $z$. Let $u$ be the
neighbor of $x$ in $G[\overline{S}]$. If $u$ has a neighbor $u'$ in
$S$, by Claim 5, then $u'$ is neither a $x_i$ nor a $y_j$. So, $u$
is forced to be colored in 1 by $u'$. Then, $x$ forces $x_1$. Now,
all vertices in $S$ are colored in 1. So, $S'$ is a zero forcing set
of $G$ and $|S'|<|S|$, which contradicts the fact that $S$ is a
minimum zero forcing set of $G$. Now we assume that $u$ has no
neighbor in $S$. Hence, $d_{G[\overline{S}]}(u)=\Delta\geq 3$ and
the number of leaves of $G[\overline{S}]-x$ is $k-1$. By the color
exchange rule, all leaves of $G[\overline{S}]-x$ but $z$ are forced
to be colored in 1 by $S'$. By Lemma 2.4, all vertices in
$\overline{S}\setminus \{x\}$ will be forced to be colored in 1, and
then $x$ forces $x_1$. This shows that $S'$ is a zero forcing set of
$G$, contradicting  the choice of $S$.

This completes the proof.

\end{proof}

\end{document}